\documentclass[11pt]{amsart}
\usepackage[utf8]{inputenc}
\usepackage{amsmath, amsthm, amsfonts, amssymb}
\usepackage[capitalise]{cleveref}
\usepackage{tikz-cd}
\usepackage{color}

\newtheorem{theorem}{Theorem}[section]
\newtheorem{lemma}[theorem]{Lemma}
\newtheorem{corollary}[theorem]{Corollary}
\newtheorem{proposition}[theorem]{Proposition}

\theoremstyle{definition}
\newtheorem{definition}[theorem]{Definition}

\theoremstyle{remark}
\newtheorem{remark}[theorem]{Remark}

\newcommand{\F}{\mathbb{F}}
\newcommand{\Z}{\mathbb{Z}}

\newcommand{\cA}{\mathcal{A}}
\newcommand{\cB}{\mathcal{B}}
\newcommand{\cC}{\mathcal{C}}
\newcommand{\cH}{\mathcal{H}}

\newcommand{\cM}{\overline{\mathcal{M}}}

\DeclareMathOperator{\Hom}{Hom}
\DeclareMathOperator{\im}{im}
\DeclareMathOperator{\Span}{Span}

\numberwithin{equation}{section}

\title[The $\Z/p$-equivariant cohomology of Deligne-Mumford space $\cM_{0, 1+p}$]{The $\Z/p$-equivariant cohomology of genus zero Deligne-Mumford space with $1+p$ marked points}
\author{Dain Kim and Nicholas Wilkins}
\date{December 2022}

\begin{document}

\maketitle

\begin{abstract}
We prove that the Serre spectral sequence of the fibration $\cM_{0, 1+p} \to E\Z/p \times_{\Z/p} \cM_{0, 1+p} \to B \Z/p$ collapses at the $E_2$ page. 
We use this to prove that: a torsion element of the $\Z/p$-equivariant cohomology with $\F_p$-coefficients of genus zero Deligne-Mumford space with $1+p$ marked points is lifted from non-equivariant cohomology.
This concludes that the only ``interesting" $\Z/p$-equivariant operations on quantum cohomology are quantum Steenrod power operations.
\end{abstract}
\section{Introduction}

\subsection{Background}

The study of equivariant cohomology operations in the context of symplectic geometry is relatively new, but it builds off much older operadic roots. The first notions of these sort of equivariant cohomology operations, specifically useful towards symplectic geometry, appeared in the work independently by Betz \cite{betz} and Fukaya \cite{fukaya}. The next appearance of such operations was by Seidel in \cite{seidel-equivaraint-pop}.

More recently, based on the initial definition by Fukaya in \cite{fukaya}, the second author and Seidel defined and studied in depth the notion of the ``Quantum Steenrod power operations" across the papers \cite{wilkins}, \cite{seidel-formal} and \cite{Seidel-Wilkins}. The idea of these papers was to define operations that look like the topological Steenrod power operations of \cite{steenrod}, but in the context of quantum cohomology: much as the Steenrod power operations measure the chain-level noncommutativity of the cup product, so do the quantum Steenrod power operations do likewise for the quantum cup product.

All of the above papers follow as their guiding principle the equivariant version of the {\it symplectic operadic principle}, i.e.: ``The nature of an equivariant symplectic invariant defined by counting holomorphic curves is determined by the operad of the space of domains."

Without delving too deeply into the technical details and definitions of the quantum Steenrod power operations, the general notion for invariants on quantum cohomology is the following: given a natural number $n$, and a permutation group $G < \text{Sym}(n)$, one can consider the Deligne-Mumford space $\cM_{0,1+n}$. This space roughly consists of the set of $1+n$ distinct marked points $(z_0,z_1,\dots,z_n)$ with $z_i \in S^2$ (with a compactification we describe in Section \ref{sec:dm-space}), up to M\"obius reparametrisations of $S^2$. Then $G$ acts via permuting the last $n$ points. The idea is, for a nice choice of $G$ and $p | |G|$ for a prime $p$, that each closed element of $C_*^G(\cM_{0,1+n}; \mathbb{F}_p)$ should have an associated equivariant quantum operation, and that any boundary in $C_*^G(\cM_{0,1+n}; \mathbb{F}_p)$ determines a $1$-dimensional cobordism (given by a $1$-dimensional parametric moduli space) between the moduli spaces determining the chain-level definitions of the two different operations. Hence, one expects that $H_*^G(\cM_{0,1+n}; \mathbb{F}_p)$ should give a good amount of information about $G$-equivariant quantum operations. One would also wonder the case $p \nmid |G|$, but in this case the $G$-equivariant cohomology of $\F_p$ coefficients is the $G$-invariant part of the ordinary cohomology, which gives no interesting operations. 

Now, the quantum Steenrod $p$-th power operation is defined using $n=p$, $G = \mathbb{Z}/p$, and the class of $H_*^{\mathbb{Z}/p}(\cM_{0,1+p}; \mathbb{F}_p)$ in question is determined by a fixed point of the $G$-action on the Deligne-Mumford space. In particular, viewing $S^2$ as the extended complex plane, say $z_0 = 0$ and $z_1,\dots,z_p$ are $p$-th roots of unity (in the obvious order $z_i = \zeta^i$ for any primitive $\zeta$).

When $p=2$, this fixed point is the entirety of $\cM_{0,1+2}$, which is just a single point. But when $p$ is larger, the Deligne-Mumford space has more topology: the question of what is the equivariant cohomology of the Deligne-Mumford space is the first step towards answering a question posed by Seidel in \cite[Section 5c]{seidel-formal}. In particular, we study $\mathbb{Z}/p$-actions as opposed to $\text{Sym}_p$-actions. Using the localization theorem e.g. \cite{Quillen}, we know that the $\mathbb{Z}/p$-equivariant cohomology of $\cM_{0,1+p}$ matches that of the fixed locus once one inverts the generator $u$ of $H^2(B \mathbb{Z}/p; \mathbb{F}_p)$. In particular, up to $u$-torsion the equivariant cohomology of $\cM_{0,1+p}$ should look like the fixed points set of the $\mathbb{Z}/p$-action, which is indeed the case where $z_0=0$ and the other $z_i$ are roots of unity, hence quantum Steenrod power operations. Another way of saying this is up to $u$-torsion, the only equivariant operations we obtain in this manner (i.e. taking classes in $H_*^{\mathbb{Z}/p}(\cM_{0,1+p}; \mathbb{F}_p)$) are quantum Steenrod power operations.

Finally, this leads us to the core question answered in this paper: ``are there any interesting $u$-torsion cycles in $H_{\mathbb{Z}/p}^*(\cM_{0,1+p}; \mathbb{F}_p)$?" A positive answer would have meant potentially interesting new operations, while a negative answer means that every $\mathbb{Z}/p$-equivariant operation is a Steenrod power operation. In this paper, we demonstrate that while there are some $u$-torsion cycles in $H_{\mathbb{Z}/p}^*(\cM_{0,1+p}; \mathbb{F}_p)$, these all correspond to non-equivariant cohomology: basically, they arise as $\mathbb{Z}/p$-invariant cycles in $H^*(\cM_{0,1+p}; \mathbb{F}_p)$. In conclusion, apart from the resulting non-equivariant operations (which are just defined via standard Gromov-Witten invariants) the only $\mathbb{Z}/p$-equivariant operations we may define are quantum Steenrod power operations. We formalise this in Corollary \ref{cor: ties_with_quantum}.

\begin{remark}
This should be contrasted with the case where one considers the composition of quantum Steenrod operations (see e.g. the quantum Adem relation \cite[Section 7]{wilkins}): in this setting, $n=p^{p}$ and $G = \mathbb{Z}/p \int \mathbb{Z}/p$ (the wreath product). At least for $p=2$ (and expected for larger $p$) there exist interesting and exotic elements of $H_{G}^*(\cM_{0,1+n}; \mathbb{F}_p)$ that do not obviously arise as Steenrod power operations.
\end{remark}

\subsection{Summary of paper}

We begin in Section \ref{sec:dm-space} with recalling some preliminary concepts: in particular, we recall Deligne-Mumford space, cohomology with local systems, a construction of $\mathbb{Z}/p$-equivariant cohomology, Serre spectral sequence, and the localization theorem.

In Section \ref{section:structure-of-Hm0n-etc}, we recall a construction of $H^*(\cM_{0,1+p};\mathbb{F}_p)$ and importantly observe how the action of $\mathbb{Z}/p$ on $\cM_{0,1+p}$ descends to an action on $H^*(\cM_{0,1+p};\mathbb{F}_p)$, noting specifically the fixed classes. We then proceed to construct the associated Serre spectral sequence associated to the fibration $$\cM_{0,1+p} \hookrightarrow \cM_{0,1+p} \times_{\mathbb{Z}/p} E \mathbb{Z}/p \rightarrow B \mathbb{Z}/p.$$

Finally, in Section \ref{sec:main-result} we proceed to prove the following, main result:

\begin{theorem} \label{thm:main}
Let $\cM_{0, 1+p}$ be the Deligne-Mumford space of genus zero with $1+p$ marked points. Then the following map is injective:
\[
H_{\mathbb Z/p}^\ast(\cM_{0,1+p} ; \mathbb F_p) \xrightarrow{\rho \oplus i^\ast} H^\ast(\cM_{0, 1+p}; \mathbb F_p) \oplus H_{\mathbb Z/p}^\ast(\cM_{0, 1+p}^{\text{fix}}; \mathbb F_p)
\]
where $\rho$ is induced by forgetting equivariant parameters and $i^\ast$ by the inclusion $i \colon \cM_{0, 1+p}^{\text{fix}} \hookrightarrow \cM_{0, 1+p}$.
\end{theorem}

To achieve this, we first demonstrate that the Serre spectral sequence collapses on the $E_2$-page. Then we use this collapse to demonstrate that the theorem holds.

\section*{Acknowledgements}
The first author is grateful to Paul Seidel for numerous helpful discussions and support. The first author was partially supported by MIT’s Undergraduate Research Opportunities Program.
The second author, funded by the Simons Foundation through award 652299, also thanks Paul Seidel. Further, the second author thanks MIT and the Max Planck Institute for Mathematics in Bonn, for hosting them respectively during the writing and the amendments to the paper.

Both authors thank the reviewer for their very helpful comments and corrections.

\section{Preliminaries}
We first give the definition of Deligne-Mumford space of genus zero, which is the space of interest. We also recall the definition of equivariant cohomology. To compute the equivariant cohomology of Deligne-Mumford space $\cM_{0, 1+p}$, we will use a fibration whose corresponding Serre spectral sequence yields this equivariant cohomology. Here, we will need cohomology with twisted coefficients, called local system. At the end of the section, we will state a localization theorem, which will be used to show the collapse of a Serre spectral sequence.

\subsection{The Deligne-Mumford space of genus zero}
\label{sec:dm-space} We give a definition of Deligne-Mumford space of genus zero with marked points.
We use the definition from \cite{mcduff-salamon}. 

A stable nodal curve of genus zero consists of finitely many Riemann spheres $S^2$ where some pairs of two spheres are joined along a nodal singularity so that the adjacency graph of Riemann spheres is a connected tree. 
An $n$-pointed stable nodal genus $0$ curve is a stable nodal genus $0$ curve with $n$ marked points on the spheres that are different from nodes so that each of the spheres has at least 3 nodes or marked points.
Two stable curves are isomorphic if there is a collection of M\"obius transformations from each Riemann sphere to another preserving nodes and marked points. Such a curve is stable in the sense that the identity is the only automorphism.

\begin{definition}
\emph{Deligne-Mumford space of genus zero with $n$ marked points} For $n \ge 3$, $\cM_{0, n}$ is the moduli space of isomorphism classes of stable curves of genus zero with $n$ marked points.
\end{definition}

Note that we are only interested in the case of $n \ge 3$, for purposes of the applications to equivariant quantum operations. It should also be noted that this $\cM_{0, n}$ is a smooth compact manifold, for $n \ge 3$.

\subsection{Cohomology with local systems}
Let $X$ be a path-connected and locally-path-connected topological space. The $n$-chains of $X$ are finite sum $\sum_i n_i\sigma_i$ where $n_i \in \mathbb Z$ and $\sigma_i : \Delta^n \to X$ is an $n$-simplex, and the space of $n$-chains is denoted as $C_n(X)$. We may extend the notion of $n$-chains by allowing coefficients $n_i$ to be elements of a fixed abelian group $G$. The cochain complex of $X$ with coefficient $G$ is defined as $\Hom(C_n(X), G)$. We call the cohomology (resp. homology) induced by such cochains (resp. chains) as cohomology with coefficients (resp. homology with coefficients).

We may extend the cohomology with coefficients further by twisting the coefficient, which leads to the notion of local system. A \emph{local system} on $X$ is a locally constant sheaf of abelian groups on $X$.

\begin{definition}
Let $\cA$ be a local system on $X$. If $\Delta^n = [v_0, v_1, \ldots, v_n]$, then the cochain complex of $X$ with local system $\cA$ is defined as 
\[
    C^n(X; \cA) = \prod_{\sigma : \Delta^n \to X} \cA(\sigma(v_0))
\]
with differential $\delta \colon C^n(X; \cA) \to C^{n+1}(X; \cA)$ given as
\[
    (-1)^n(\delta c)(\sigma) = \cA(\gamma)^{-1} c(\partial_0 \sigma) + \sum_{i=1}^{n+1} (-1)^i c(\partial_i \sigma)
\]
for $c \in C^n(X; \cA)$ and $\sigma : \Delta^{n+1} \to X$
where $\gamma$ is a path $t \mapsto \sigma((1-t)v_0 + tv_1)$ and $\cA(\gamma)$ is the map from the stalk at $\sigma(v_0)$ to the stalk at $\sigma(v_1)$ induced by $\gamma$. 
The \emph{cohomology of $X$ with local system $\cA$} is $H^\ast(X;\cA) := \ker \delta / \im \delta$.
\end{definition}
We refer the reader to \cite[Chapter 4]{Whitehead} for more details.
The local coefficient $G$ may be viewed as the local system that is the trivial bundle $X \times G$.

Even though the local system itself is complicated, if the local system is decomposable as a finite direct sum of local systems, then the cohomology can be written as the direct sum of cohomology with each component of the local system.

\begin{proposition} \label{prop: decompose-local-system}
Let $X$ be a path-connected and locally-path-connected topological space and $F$ a field.
If a local system $\cA$ is decomposed as $\cA = \cB \oplus \cC$ as vector spaces over $F$ for local systems $\cB$ and $\cC$ on $X$, then 
\[
    H^\ast(X;\cA) = H^\ast(X;\cB) \oplus H^\ast(X;\cC).
\]
\end{proposition}
\begin{proof}
Since $\cA = \cB \oplus \cC$, 
\begin{align*}
    C^\ast(X; \cA) &= C^\ast(X; \cB) \oplus C^\ast(X; \cC),\\
    \cA(\gamma) &= \cB(\gamma) \oplus \cC(\gamma)
\end{align*}
canonically. Therefore, $\delta_\cA = \delta_\cB \oplus \delta_\cC$ for differentials corresponding to local systems, which leads to the desired direct decomposition of cohomology.
\end{proof}

\subsection{$\Z/p$-equivariant cohomology}

We begin by defining the classifying space $B G$, for some group $G$. First we consider $E G$, which is a contractible space equipped with a free $G$-action. In our case, with $G = \Z/p$, we may pick $E \Z/p = S^{\infty} \subset \mathbb{C}^{\infty} = \bigcup_j \mathbb{C}^j$, taking the $\Z/p$-action to be diagonal multiplication by $e^{2 \pi i / p}$. In general, the classifying space is $$B G := E G / G.$$ Then, given some topological space $X$ equipped with a $G$-action $\sigma$, we define its homotopy quotient $EG \times_G X$ to be the quotient of $EG \times X$ by the diagonal action (abusively) denoted $$\sigma(g) \cdot (v,x) = (e^{2 \pi i / p} v, \sigma(g) \cdot x),$$ for $g \in G$.

\begin{definition}
Let $X$ be a topological space, and a topological group $G$ acts continuously on $X$. Then \emph{the equivariant cohomology of $X$ with action $G$ and a coefficient $F$} is
\[
    H_G^\ast(X; F) := H^\ast(EG \times_G X ; F)
\]
where $EG$ is the universal cover of the classifying space $BG$.
\end{definition}

We notice further that, in the case that $X$ is fixed by the $G$-action, then the homotopy quotient $EG \times_G X = BG \times X$, and hence one can use the K\"unneth isomorphism to write $$H_G^\ast(X; F) = H^\ast(EG \times_G X ; F) \cong H^\ast (BG; F) \otimes H^\ast(X;F).$$

We are interested in the case where $G = \Z/p$. Indeed, noting that  $\cM_{0, 3}$ is a single point (and therefore has no torsion equivariant cohomology), we may restrict our attention to $n \ge 4$: thus, in this paper we will assume that $p>2$. 

When $F = \mathbb{F}_p$, we follow the conventions of \cite[Equation (5.5)]{seidel-formal} (with the natural extension for general coefficient systems). In particular, we fix some generator $g$ of $\mathbb{Z}/p$ and write $$C^*_{\mathbb{Z}/p}(X;\mathbb{F}_p) =  \mathbb{F}_p [u] \otimes \Lambda(e) \otimes C^*(X; \mathbb{F}_p),$$ where $|e| = 1$ and $|u|=2$, with $$\begin{array}{lll} d ( u^k \otimes c) &=& e u^k \otimes (gc - c) + u^k \otimes dc \\ d(u^k e \otimes c) &=&   u^{k+1} \otimes (c + gc + \dots + g^{p-1} c) - u^k e \otimes dc. \end{array}$$ 

\subsection{Serre spectral sequence}

\begin{theorem}[{\cite[Proposition 6]{Serre}}] \label{thm:Serre_sseq}
If $F \to E \xrightarrow{\pi} B$ is a Serre fibration and $G$ a fixed abelian group, there is a spectral sequence
\[
    E^{p,q}_2 = H^p(B; \cH^q(F;G)) \Rightarrow H^\ast(E;G),
\]
which is called a Serre spectral sequence. Here, $\cH^q(F;G)$ is the local system over $B$ with stalk $H^q(F;G)$ induced by the Serre fibration $\pi$. If $B$ is a CW complex, denote by $B^k$ the $k$-skeleton of $B$. We have a filtration of $E$ as $E^k = \pi^{-1}(B^k)$.
Then the filtration of $H^m(E;G)$ is given as $F^k H^m(E;G) = \ker( H^m(E;G) \to H^m(E^{k-1};G))$.
\end{theorem}
For brevity, we define $F^m_k := F^k H^m(E; G)$.

The \cref{thm:Serre_sseq} shows that the $E^2$ page of the Serre spectral sequence corresponding to a fibration $X \hookrightarrow EG \times_G X \to BG$ with coefficient $\F_p$ is
\[
    H^i(BG ; \cH^j(X ; \F_p))
\]
and converges to $H^\ast(EG \times_G X; \F_p) = H^\ast_G(X;\F_p)$.

\subsection{Localization} The intuitive notion of the localization theorem is that, up to $u$-torsion, all the equivariant cohomology classes of a space $X$ are encoded in the fixed-point space $X^{\text{fix}}$. We will use the localization theorem on the Deligne-Mumford space $\cM_{0, 1+p}$. The fixed-point space of the Deligne-Mumford space consists of a set of points, and so the localization theorem provides a powerful restriction on the equivariant cohomology of $\cM_{0, 1+p}$.

We restate the localization theorem given by Quillen as a version we will use to the Deligne-Mumford space $\cM_{0, 1+p}$ with a group $\Z/p$.
\begin{theorem}[{\cite[Theorem 4.2]{Quillen}}] \label{thm:localization}
For a compact manifold $X$ with an action of $\Z/p$, the inclusion of the fixed-point set $X^\text{fix} \hookrightarrow X$ induces an isomorphism 
\[
    H_{\Z/p}^\ast (X; \F_p)[u^{-1}] \xrightarrow{\cong} H_{\Z/p}^\ast(X^\text{fix}; \F_p)[u^{-1}]  
\]
where $u$ is a generator of $H^2(B\Z/p; \F_p)$.
\end{theorem}

\section{The structure of $H^*(\cM_{0, 1+p};\mathbb{F}_p)$ and the Serre spectral sequence}
\label{section:structure-of-Hm0n-etc}
In this section, we investigate how $\Z/p$ acts on $H^\ast(\cM_{0, 1+p};\F_p)$, which will be used to compute Serre spectral sequence. We also compute $\cM_{0, 1+p}^{\Z/p}$, which will appear in the localization theorem.
\subsection{The cohomology of $\cM_{0, 1+p}$ and its fixed point under $\Z/p$ action}

The cohomology ring structure of $\cM_{0, 1+p}$ was completely described by Keel in his paper \cite{Keel}, and later an equivalent description but in symmetric basis was given by \cite{Etingof}. We will use the basis by \cite{Etingof} to get a benefit of its symmetry.
\begin{theorem}[{\cite[Theorem 5.5]{Etingof}}] \label{thm:basis_of_DM} For the Deligne-Mumford space $\cM_{0, 1+p}$,
let the marked points be $x_1, \ldots, x_{p+1}$ and $X = \{x_1, \ldots, x_p\}$.
For each $S \subset X$ with $|S| \ge 3$, there exists a cohomology class $\Pi_S \in H^\ast(\cM_{0, 1+p};\Z)$ so that monomials of the form $\prod_{|S| \ge 3} \Pi_S^{d_S}$ satisfy the following conditions freely generate $H^\ast(\cM_{0, 1+p}; \Z)$ as a $\Z$-module.
\begin{itemize}
\item If both $d_S$ and $d_T$ are positive, then $S \cap T \in \{\emptyset, S, T\}$.
\item For each $S$ with $d_S > 0$, if $S_1, \ldots, S_k$ are disjoint and form the maximal proper subsets of $S$ with $d_{S_i}>0$, then
\[
    d_S < k - 1 + |S| - \sum_i |S_i|.    
\]
\end{itemize}
Moreover, if $\sigma \colon X \to X$ is an action induced by $1 \in \Z/p\Z$, i.e. $\sigma(x_i) = x_{i+1}$ for $1 \le i \le p-1$, $\sigma(x_p) = x_1$, and $\sigma(x_{p+1}) = x_{p+1}$, 
then $\sigma(\Pi_S) = \Pi_{\sigma(S)}$ and is compatible with the ring structure of $H^\ast(\cM_{0, 1+p};\Z)$.
\end{theorem}

In the following corollary, we adapt the aforementioned basis for $H^\ast(\cM_{0, 1+p};\Z)$ and find the fixed classes under $\Z/p$ action.

\begin{corollary} \label{cor: number_of_fixed_generators}
There is a basis for $H^\ast(\cM_{0, 1+p};\F_p)$, on which $\sigma^\ast$ acts, which contains exactly $p-1$ fixed elements (and the rest partitioning into $\Z/p$ orbits of size $p$).
\end{corollary}
\begin{proof}
We take the basis of $H^\ast(\cM_{0, 1+p}; \Z)$ described in \cref{thm:basis_of_DM}.
Suppose that $\prod_{|S| \ge 3} \Pi_S^{d_S}$ is fixed under $\sigma$, and $d_T>0$ for some $T$. 
Let $1 \le i < j \le p$ be two elements in $T$.

Since the monomial $\prod_{|S| \ge 3} \Pi_S^{d_S}$ is fixed under $\sigma$ and $d_T > 0$, it follows that $d_{\sigma(T)} > 0$. Repeating the argument, we also have $d_{\sigma^{j-i}(T)} > 0$.

Note that as $i \in T$, by shifting the labels $j-i$ times, $j \in \sigma^{j-i}(T)$ and $|T| = |\sigma^{j-i}(T)|$, so by the first condition in \cref{thm:basis_of_DM}, $\sigma^{j-i}(T) = T$.
But as $p$ is a prime and $0<j-i < p$, i.e. $j-i$ is invertible in $\mathbb{F}^{\times}_p$, we obtain $\sigma(T) = T$, which implies that $T=X$.
Then the second condition of \cref{thm:basis_of_DM} forces that $1 \le d_X \le p-1$. In the other direction, the cohomology classes $\Pi_X^k$ are fixed since $\sigma(X) = X$.

Note that any other monomials form several $p$-cycles since for any $S \subset X$, $\sigma^p(S) = S$. This proves the corollary for $H^\ast(\cM_{0, 1+p} ; \Z)$, i.e. the set of monomials is decomposed into $p-1$ fixed points and some $p$-cycles.

Since $H^\ast(\cM_{0, 1+p}; \Z)$ is free, the generators of $H^\ast(\cM_{0, 1+p} ; \F_p)$ are obtained by reducing the coefficients of the generators in mod $p$. There are no relations among the generators because otherwise we can lift the relation to $H^\ast(\cM_{0,1+p}; \Z)$. Furthermore, the $\mathbb Z/p$ action acts in the same way, so the desired result follows.
\end{proof}

To use the localization theorem, we investigate the fixed points of $\cM_{0, 1+p}$ under the $\Z/p$ action. To be explicit, elements of $\cM_{0, 1+p}$ consist of a nodal Riemann surface $S$ along with a $(1+p)$-tuple $(z_1,\dots,z_p, z_{p+1}) \subset S^{p+1}$, and the $\Z/p$-action acts on the indices of $z_1,\dots,z_p$, i.e. $n \cdot z_j = z_{j + n \text{ mod } p}$, while fixing $z_{p+1}$. 

\begin{proposition}
\label{proposition:number-fixed-points}
There are exactly $p-1$ many fixed points of $\cM_{0, 1+p}$ under the $\Z/p$ action.
\end{proposition}
\begin{proof}
We denote by $\cM_{0, 1+p}^{\Z/p}$ the fixed locus under the $\Z/p$-action. Further, the marked points on a stable curve $C \in \cM_{0, 1+p}^{\Z/p}$ are denoted by $x_1, \ldots, x_{p+1}$.

We first show that $C$, as a nodal curve, has exactly one $\mathbb C P^1$ component.
Suppose for a contradiction that $C$ is composed of more than one sphere. The point $x_{p+1}$ must then be alone, because if $x_j$ is on the same sphere then so too should all of the other $x_i$, which would imply $C$ has only one Riemann sphere. Since the number of marked points on the sphere containing $x_1$ is preserved throughout the action of $\Z/p$, call this number $b$, then for any sphere that contains a marked point in $\{ x_1,\dots,x_p \}$, that sphere must have exactly $b$ marked points from that set. Denote by $c$ the number of such spheres containing a marked point in $\{ x_1,\dots,x_p \}$. Since $p$ is a prime, and $p = b \cdot c$, there are 2 cases ($b=1$ or $c=1$): each sphere has at most one marked point or $x_1, \ldots, x_p$ are on one sphere and $x_{p+1}$ is on another sphere.
But both cases lead to a contradiction because both lead to trees of spheres with leafs containing less than $2$ marked points (hence less than $3$ special points overall). Hence $C$ has exactly one Riemann sphere.

Therefore, $x_1, x_2, \ldots, x_{p+1}$ lie on one Riemann sphere. By applying a M\"obius transformation, we may assume that $x_1 = 1, x_2 = \eta, x_{p+1} = 0$ for a root of unity $\eta$ with $\eta^p = 1$. Let $\sigma$ be the action of $1 \in \Z/p$. Then
\begin{align}
    \sigma x_{p+1} &= x_{p+1} = 0 \label{eqn:fix_xp+1} \\
    \sigma x_1 &= x_2 = \eta \label{eqn:fix_x1} \\
    \sigma^{p-1}x_2 &= x_1 = 1. \label{eqn:fix_x2}
\end{align}
Since $C$ and $\sigma C$ are isomorphic, there is a M\"obius transformation $z \mapsto \frac{az+b}{cz+d}$ which gives an isomorphism between two stable curves.
Then \eqref{eqn:fix_xp+1} gives $b = 0$ and by \eqref{eqn:fix_x1} we may assume that $a = \eta$ and $c+d = 1$.
The M\"obius transformation given by $\sigma^{p-1}$ is a rational function whose numerator is constant and denominator is of degree $p-1$ polynomial in $c$. This can be proven by induction on the exponent of $\sigma$. Therefore, applying \eqref{eqn:fix_x2} to this rational function, a solution $\sigma$ satisfying \eqref{eqn:fix_xp+1}, \eqref{eqn:fix_x1} and \eqref{eqn:fix_x2} is in fact a solution of a degree $p-1$ polynomial over $\mathbb{C}$, hence there are at most $p-1$ many such M\"obius transformations. Every fixed point of $\cM_{0,1+p}$ must have that the induced $\mathbb{Z}/p$ action is such a M\"obius action, and any such action completely determines the points $x_3,\dots,x_p$. Hence, $p-1$ is similarly an upper bound of the number of fixed points of $\cM_{0, 1+p}$ under the $\Z/p$ action.

Now we describe these $p-1$ fixed points. Those points are given as $x_{p+1} = 0$, $x_1 = 1$, and $x_k = \eta^{s(k-1)}$ for $1 \le s \le p-1$ and $2 \le k \le p+1$. Denote the stable curve by $C_s$.
Indeed, they are fixed under $\Z/p$ since M\"obius transformation $z \mapsto \eta^s z$ maps $C$ to $\sigma C$.
To show that they represent distinct isomorphism classes, suppose for a contradiction that there is a M\"obius transformation $z \mapsto \frac{az+b}{cz+d}$ from $C_{s_1}$ to $C_{s_2}$ for $1 \le s_1 < s_2 \le p-1$.
There exists an integer $2 \le r \le p-1$ with $rs_1 \equiv s_2 \pmod p$. The M\"obius transformation maps the marked points of $C_{s_1}$ to the marked points of $C_{s_2}$ and so does the map $z \mapsto z^r$. Therefore, 
\[
\frac{az+b}{cz+d} = z^r
\]
has at least $p+1$ solutions $z = 0, 1, \eta^{s_1}, \ldots, \eta^{s_1(p-1)}$. But as $r \le p-1$, the equation is equivalent to a polynomial of degree at most $p$, which forces the map to be the identity. Therefore, $s_1 = s_2$, a contradiction.
\end{proof}

\subsection{The $E_2$ page of the 
Serre spectral sequence}
We compute the $E_2$ page of the Serre spectral sequence for the fibration 
\[
\cM_{0, 1+p} \to E\Z/p \times_{\Z/p} \cM_{0, 1+p} \to B \Z/p.
\]

\begin{proposition} \label{prop: trivial_representation}
If $A$ is the trivial $\Z/p$-representation of $\F_p$, then 
\[
    H^i(B \Z/p; A) \cong \F_p
\] 
for all $i$.
\end{proposition}
\begin{proof}
\cite[Section 2a]{Seidel-Wilkins} gives a cell decomposition of $E\Z/p$ with coefficient $\F_p$ as follows. There is $\Delta^i \in C^i(E\Z/p; \F_p)$ such that $\Delta^i, \sigma \Delta^i, \ldots, \sigma^{p-1}\Delta^i$ freely span $C^i(E\Z/p; \F_p)$ where $\sigma$ is the action of $1 \in \Z/p$, and the differential is given as 
\begin{align}
    \delta \Delta^i = \begin{cases}
        \sigma \Delta^{i+1} - \Delta^{i+1} & i \text{ even},\\
        \Delta^{i+1} + \sigma \Delta^{i+1} + \cdots + \sigma^{p-1}\Delta^{i+1} & i \text{ odd}.
    \end{cases} \label{eqn: cell_decomposition_of_EZ/p}
\end{align}
Therefore each $H^i(B\Z/p;A)$ is spanned by $\Delta^i$, and $\delta \Delta^i = 0$.
Since every differential is zero, no $\Delta^i$ is exact. Hence, $H^i(B\Z/p;A) \cong \F_p$.
\end{proof}

\begin{proposition} \label{prop: p-dim_representation}
If $A$ is a $p$ dimensional $\F_p$-vector space $V = \Span(v_1, \ldots, v_p)$ with the induced $\Z/p$-action on indices, then
\[
    H^i(B\Z/p ; A) \cong \begin{cases}
        \F_p &i=0,\\
        0 &i>0.
    \end{cases}
\] 
\end{proposition} 
\begin{proof}

This is an immediate consequence of the Eckmann-Shapiro lemma, see \cite{eckmann1953cohomology}, relating $$H^*(B\Z/p; A) \cong H^*(\Z/p;A) \cong H^*(1; \mathbb{F}_p)$$ (noting that $A$ is the regular representation of $\Z/p$). 
\end{proof}

\begin{lemma} \label{lem: E2-page-sseq}
The $E_2$ page of the Serre spectral sequence for the fibration 
\[
\cM_{0, 1+p} \to E\Z/p \times_{\Z/p} \cM_{0, 1+p} \to B \Z/p
\]
is 
\[
    E_2^{i, j} \cong \begin{cases}
        \F_p^{r_j} \times \F_p & i=0, \  0 \le j \le 2(p-2), \ j \text{ even},\\
        \F_p & i \ge 1, \ 0 \le j \le 2(p-2), \ j \text{ even},\\
        0 & \text{otherwise}.
    \end{cases}  
\]

where $r_j$ is the number of copies of the regular representation in cohomological degree $j$, appearing in the local system associated to $H^j( \cM_{0, 1+p} ; \F_p )$ over $B \Z/p$.

In particular, as a ring, $E_2^{i,j}$ is generated by: \[ \begin{array}{l} 1 \otimes \alpha \in E_2^{0,2}, \\
u \otimes 1 \in E_2^{2,0}, \\
e \otimes 1 \in E_2^{1,0},\\ \text{some further additional elements of } E_2^{0,\ast}, \end{array} \] subject to at least the following two relations: 
\begin{itemize}
\item these additional elements of $E_2^{0,\ast}$ are eliminated by multiplication by $u \otimes 1$ and $e \otimes 1$, 
\item $e \otimes 1$ is eliminated by multiplication with $e \otimes 1$.
\end{itemize}
\end{lemma}
\begin{proof}
By \cref{thm:Serre_sseq},
\begin{align*}
    E_2^{i, j} = H^i(B\Z/p; H^j( \cM_{0, 1+p} ; \F_p )) &\Rightarrow H^\ast(E\Z/p \times_{\Z/p} \cM_{0, 1+p} ; \F_p)\\
    &\quad = H_{\Z/p}^\ast (\cM_{0, 1+p} ; \F_p).  
\end{align*}
\cref{cor: number_of_fixed_generators} breaks down the local system into $\Z/p$ representations.
\cref{prop: decompose-local-system} then tells one that we only need to consider \cref{prop: trivial_representation} and \cref{prop: p-dim_representation}. 

The ring structure follows because \cref{prop: decompose-local-system} holds for cohomology with coefficients in a local system of rings, and the proof of \cref{cor: number_of_fixed_generators} tells us exactly that the splitting of $H^{\ast}(\cM_{0, 1+p} ; \F_p )$ as $\Z/p$-representations is also a splitting as rings.
\end{proof}

\section{Main Result}
\label{sec:main-result}
\subsection{The spectral sequence collapses}
\begin{theorem} \label{thm:degeneracy_of_sseq}
The Serre spectral sequence for the fibration $$\cM_{0, 1+p} \to E\Z/p \times_{\Z/p} \cM_{0, 1+p} \to B \Z/p$$ collapses at the $E_2$ page.
\end{theorem}
\begin{proof}
\cref{lem: E2-page-sseq} gives the $E_2$ page of the spectral sequence. In the notation of the previous sections, we denote $$H^\ast(B\Z/p ; \F_p) = \F_p[u] \otimes \Lambda[e], \text{ and } H^\ast(\cM_{0, 1+p}; \F_p)^{\text{fix}} = \F_p[\alpha]/(\alpha^{p-1}),$$ where $\deg \alpha = 2$. Then, when $i> 0$, one can write elements of $E_2^{i,j}$ of the form $e^{\epsilon} u^l \otimes \alpha^{\lambda}$. We will abusively write this as $\alpha^{\lambda} e^{\epsilon} u^l$.

Suppose for a contradiction that the spectral sequence does not collapse at the $E_2$ page.

\textit{Case 1.} There is some nonvanishing differential on $\alpha^{\lambda} e^{\epsilon} u^l$, where $\epsilon \in \{0,1\}$. 

In particular, there is some minimal $r$ and a nontrivial differential $d^r$ on some page $E_r$ such that $d^r(\alpha^\lambda u^i) \ne 0$ or $d^r(\alpha^\lambda e u^i) \ne 0$ for some $i$. If $d^r(\alpha^\lambda e u^i) \ne 0$, then we have 
\[
d^r(\alpha^\lambda e u^i) = d^r(\alpha^\lambda u^i)e + (d^re)\alpha^\lambda u^i = d^r(\alpha^\lambda u^i) e \ne 0,
\]
which implies that $d^r(\alpha^\lambda u^i) \ne 0$. Hence, we may only consider the case $d^r(\alpha^\lambda u^i) \ne 0$.
Since 
\[
d^r(\alpha^\lambda u^{i+j}) = d^r(\alpha^\lambda u^i)u^j + \alpha^\lambda u^id^r(u^j) = d^r(\alpha^\lambda u^i)u^j
\]
and $u$ acts injectively on the $E_2, \ldots, E_r$ pages except for those generators arising from $p$-cycles of generators by \cref{prop: p-dim_representation}, this implies that $d^r(\alpha^\lambda u^{i+j}) \ne 0$.
Hence, there are infinitely many nontrivial differentials.
In particular, for some sufficiently large $m$, there is an infinite sequence $\{ m+2k \}_{k \ge 0}$ such that:
\[
    p-1 = \sum_{i+j = m+2k} \dim_{\F_p} E_2^{i,j} > \sum_{i+j = m+2k} \dim_{\F_p} E_\infty^{i,j} = \dim_{\F_p} H_{\Z/p}^{m+2k} (\cM_{0, 1+p}; \F_p).
\]

We write $m = 2n+ \epsilon$ with $\epsilon \in \{0,1 \}$.

Now consider $\iota^\ast \colon H_{\Z/p}^m(\cM_{0, 1+p};\F_p) \to H_{\Z/p}^m(\cM_{0, 1+p}^{\text{fix}};\F_p)$.
We know that once we invert $u$ this becomes an isomorphism by \cref{thm:localization}. Denote this isomorphism by \begin{equation} \label{equation:u-inverted-iota} \iota^\ast[u^{-1}] : H_{\Z/p}^\ast(\cM_{0, 1+p};\F_p)[u^{-1}] \to H_{\Z/p}^\ast(\cM_{0, 1+p}^{\text{fix}};\F_p)[u^{-1}].\end{equation}

However, for the infinite sequence of $\{ m +2k \}_{k \ge 0}$ described above, the $\iota^\ast$ cannot be surjective due to the fact that $\dim H_{\Z/p}^m(\cM_{0, 1+p};\F_p) < p-1$, but $$\dim H_{\Z/p}^m(\cM_{0, 1+p}^{\text{fix}};\F_p) = \dim ( H^0(\cM_{0, 1+p}^{\text{fix}};\F_p) \otimes H^m(B \Z/p; \F_p) ) = p-1.$$ Here we use the K\"unneth isomorphism for the first equality, and Proposition \ref{proposition:number-fixed-points} for the second (i.e. $\dim H^0(\cM_{0, 1+p}^{\text{fix}};\F_p) = p-1$).

By the pigeonhole principle, there exists $v \in H^0(\cM_{0, 1+p}^{\text{fix}};\F_p) \cong \F_p^{\oplus (p-1)}$, and a set of strictly increasing integers $n_1,n_2,\ldots$ such that $v \otimes e^{\epsilon} u^{n+n_i} \in H^{m+2n_i}(\cM_{0, 1+p}^{\text{fix}};\F_p)$ is not hit by $\iota^*$ for each $n_i$. 

Now, considering $v \otimes e^{\epsilon} u^n$ as an element of $ H^{*}(\cM_{0, 1+p}^{\text{fix}};\F_p)[u^{-1}]$, we see  that it must be hit by some element of $H^*_{\Z/p}(X)[u^{-1}]$ under $\iota[u^{-1}]$, as \eqref{equation:u-inverted-iota} is an isomorphism. 
 
 Hence, there is some $\sum_j a_ju^{-j} \in H_{\mathbb Z/p}^*(\cM_{0, 1+p}; \F_p)[u^{-1}]$ with $a_j \in H_{\mathbb Z/p}^{m+2j}(\cM_{0, 1+p};\F_p)$ such that 
\[
    \iota^\ast[u^{-1}](\sum_j a_ju^{-j}) = v \otimes e^\epsilon u^n.  
\]
Let $J$ be maximal such that $a_J$ is nonzero. Then if $k > J$, we obtain that $\sum_{j} a_ju^{k-j} \in H_{\mathbb Z/p}^{m+2k}(\cM_{0, 1+p};\F_p)$, and in $H_{\mathbb Z/p}^{m+2k}(\cM_{0, 1+p};\F_p)[u^{-1}]$ we see that $\sum_{j} a_ju^{k-j} = (\sum_j a_j u^{-j})u^k$. Then for $\nu \colon H_{\Z/p}^\ast(\cM_{0, 1+p}^{\text{fix}}) \hookrightarrow H^\ast_{\Z/p}(\cM_{0, 1+p}^{\text{fix}})[u^{-1}]$ and $\eta \colon H^\ast_{\Z/p}(\cM_{0, 1+p}) \to H^\ast_{\Z/p}(\cM_{0, 1+p})[u^{-1}]$,

$$\begin{array}{lll}
    \nu \iota^\ast(\sum_j a_ju^{k-j}) &=& \iota^\ast[u^{-1}]\eta(\sum_j a_ju^{k-j}) \\  &=& \iota^\ast[u^{-1}]((\sum_j a_ju^{-j})u^k)  \\&=& \iota^\ast[u^{-1}](\sum_j a_ju^{-j})u^k \\ &=& v \otimes e^{\epsilon_i}u^{n_i+k}.  
\end{array}$$
Now we choose $k>J$ such that $k = n_l$ for some $l$. This yields a contradiction, because $\iota^\ast$ does not surject onto $v \otimes e^\epsilon u^{n+k} = v \otimes e^\epsilon u^{n+n_l}$.

\textit{Case 2.} There is some nonvanishing differential on some element of the spectral sequence arising from $p$-cycles in $E^{0,j}_2$ (the additional generators in the statement of \cref{lem: E2-page-sseq}).

There is some minimal $r$ and a nontrivial differential $d^r$ on $E_r$ such that $d^r(x) \ne 0$ for $x \in \F_p^{\text{\# of $p$-cycles}} \le E_2^{0,j}$.
Note that $xu = 0$. But since $u$ acts injectively apart from such $x$, we deduce $0 = d^r (0) = d^r (xu) = d^r (x)u$ is nonzero, a contradiction.

From the above two cases, we conclude that the spectral sequence collapses at $E_2$ page.
\end{proof}

We now have all of the required components to prove the main theorem.

\subsection{The proof of \cref{thm:main}}
Since the spectral sequence collapses at the $E_2$ page, we have
\[
    \phi \colon H_{\Z/p}^m(\cM_{0, 1+p}; \F_p) \xrightarrow{\sim} \bigoplus_{i+j=m} E_2^{i,j},
\]
an isomorphism of $\F_p$ vector spaces.

We assume, for a contradiction, that there exists some nonzero $x \in H_{\Z/p}^m(\cM_{0, 1+p} ; \F_p)$ such that $\rho(x) = 0$ and $\iota^\ast(x) = 0$.
Since $\phi$ is an isomorphism, there is a decomposition of $x = x_0+x_1+\cdots+x_m$ so that $\phi(x_i) \in E_2^{i, m-i}$.

We first prove that $x_0 = 0$. 
By the second statement of \cref{thm:Serre_sseq},
\begin{align*}
    F^0 H_{\Z/p}^m( \cM_{0, 1+p} ) &= H^m( E\Z/p \times_{\Z/p} \cM_{0, 1+p} ),\\
    F^1 H_{\Z/p}^m( \cM_{0, 1+p} ) &= \ker( H_{\Z/p}^m(\cM_{0, 1+p}) \to H^m(\cM_{0, 1+p}) ).
\end{align*}
Hence the map $\phi_0 \colon H_{\Z/p}^m (\cM_{0, 1+p}) \to E_2^{0, m}$, the projection of $\phi$ onto $E_2^{0, m}$, is given as
\[
    H_{\Z/p}^m(\cM_{0, 1+p}) \to H_{\Z/p}^m(\cM_{0, 1+p}) / \ker( H_{\Z/p}^m(\cM_{0, 1+p}) \to H^m(\cM_{0, 1+p} ) )
\]
induced by the projection. 
Let $p \colon (E\Z/p \times_{\Z/p} \cM_{0, 1+p}, \cM_{0, 1+p}) \to (B\Z/p, \ast)$.
\[
\begin{tikzcd}
H_{\Z/p}^m(\cM_{0, 1+p}) \arrow[r, "\rho"] &H^m(\cM_{0, 1+p}) \arrow[r, "\delta"] \arrow{rrd}[near end]{d^m} & H^{m+1} (E\Z/p \times_{\Z/p} \cM_{0, 1+p}, \cM_{0, 1+p})\\
&& H^{m+1}(B\Z/p, \ast) \arrow{u}[near start]{p^\ast} \arrow[r] &H^{m+1}(B\Z/p).
\end{tikzcd}  
\]
Since the domain of $d^m$ is $H^m(\cM_{0, 1+p})$, the transgression gives that for every $\alpha \in H^m(\cM_{0, 1+p})$ there exists $\beta \in H^{m+1}(B\Z/p, \ast)$ with $\delta \alpha = p^\ast \beta$ and $d^m \alpha = \tilde{\beta} \in H^{m+1}(B\Z/p)$, where $\tilde{\beta}$ is the image of $\beta$ in $H^{m+1}(B\Z/p, \ast) \to H^{m+1}(B\Z/p)$.
But as $d^m$ vanishes, $\beta = 0$, and hence $\delta = 0$.
Then by the exactness of the first row, $\rho$ is surjective, and therefore $\rho$ induces the same map as $\phi_0$.
In particular, $\phi(x_0) = \phi_0(x_0) = \rho(x) = 0$, and hence $x_0 = 0$.

Now we prove that $x_i = 0$ for all $i \ge 1$. Let $j \ge 1$ be the minimal index such that $x_j \ne 0$. Let $u$ be the cohomology class as in the proof of \cref{thm:degeneracy_of_sseq}, and $v \in F_2^2/\langle 0 \rangle$ the class of $u$ on the level of the spectral sequence. Let $\phi_k^m \colon H_{\Z/p}^m (\cM_{0, 1+p}) \to E_2^{k,m-k}$ be the projection map of $\phi$ onto $E_2^{k, m-k}$.

Since the cup product induces $F_j^m \times F_{2k}^{2k} \to F_{j+2k}^{m+2k}$ and this product induces $F_j^m/F_{j+1}^m \times F_{2k}^{2k} \to F_{j+2k}^{m+2k}/F_{j+2k+1}^{m+2k}$, we obtain
$\phi_j^m(x_j)v^k = \phi_{j+2k}^{m+2k}(x_ju^k)$. 
But as in the proof of \cref{thm:degeneracy_of_sseq}, $\phi(x_j)v^k \ne 0$, and hence $\phi_{j+2k}^{m+2k}(x_j u^k) \ne 0$. In the case where $s>j$, we can run a similar argument using the inclusion $(F_s^m, F_{s+1}^m) \rightarrow (F_j^m, F_{j+1}^m)$ and considering $F_s^m/F_{s+1}^m \times F_{2k}^{2k} \to F_{j+2k}^{m+2k}/F_{j+2k+1}^{m+2k}$. In particular, we have $\phi_{j+2k}^{m+2k}(x_su^k) = \phi_j^m(x_s) v^k = 0$. Therefore, $\phi_{j+2k}^{m+2k}(xu^k) \ne 0$.  
In particular, $xu^k \ne 0$ for all $k$. Then by \cref{thm:localization}, $\iota^\ast(x) \ne 0$, which is a contradiction.
Therefore, $x = 0$, and the desired result follows.

\subsection{The connection with quantum operations}

In this section, we will describe the relevance of this topological result to symplectic geometry. In order to circumvent the usual requirements of long technical definitions, we will instead provide citations to the relevant technical work for the interested reader. 

Fix some prime $p$. Suppose that $(M,\omega)$ is a symplectic manifold, with appropriate technical conditions to define quantum cohmology. As mentioned above, we will not give a formal definition of quantum cochains $QC^*(M)$ or quantum cohomology $QH^*(M)$, nor the conditions necessary to define them, but we refer the reader to the definition of quantum cohomology in \cite{mcduff-salamon}.

\begin{definition}
We will say that, given a parametrised moduli space of holomorphic maps $\mathcal{M}$, this moduli space is {\it $G$-equivariantly parametrised of order $a$} if:
\begin{itemize}
\item there is a $G$-action on $\{1,\dots,a\}$, inducing an action on the latter $a$ marked points of $\overline{\mathcal{M}}_{0,1+a}$ (the $1+a$ marked points being denoted by $(m_0,m_1,\dots,m_a)$),
\item there is some $i \in \mathbb{Z}_{\ge 0}$ and some finite dimensional smooth manifold $A \subset \overline{\mathcal{M}}_{0,1+a} \times_{G} EG $ representing a homology class in $H_*^{G}(\overline{\mathcal{M}}_{0,1+a})$
\item there is some collection of Hamiltonians $H_w : M \rightarrow \mathbb{R}$ and almost complex structures $J_w$ on $M$ (for $w \in EG$), with conditions such that a generic choice of these ensures regularity of the moduli space (e.g. if $M$ is monotone, we can require $H_w \cong 0$ and $J$ to be independent of $w$, whereas if $M$ is weakly monotone we must ensure $J$ is independent of $w$ by $H_w \neq 0$. See \cite[Section 4]{Seidel-Wilkins} or \cite[Section 4]{wilkins} for more details),
\item there is some $\beta \in QH^*(M) \otimes H^*_{G}(QC^*(M)^{\otimes a}) $, represented by $B \subset M \times \Pi_1^a M \times EG$ (where $B$ is made as a choice among generic representatives such that regularity of the moduli space holds),
\end{itemize}

 such that the moduli space consists of triples $(m,w,u)$ satisfying:
\begin{enumerate}
    \item $[m,w] \in A$, where here $[\bullet]$ denotes the equivalence class with respect to the $G$-quotient
    \item $u : \text{Und}(m) \rightarrow M$, where here $\text{Und}(m)$ is the underlying nodal genus $0$ Riemann surface 
    \item $\overline{\partial}_{J_w} u = X_{H_w}$,
    \item $(u(m_0), u(m_1), \dots, u(m_n), w) \in B$.
\end{enumerate}

\end{definition}

We will define an {\it operadic $G$-equivariant quantum operation of order $a$} to be some additive homomorphism $$QH^*(M; \mathbb{F}_p) \rightarrow H^*_{G}(QC^*(M)^{\otimes a}; \mathbb{F}_p) \rightarrow QH^*(M; \mathbb{F}_p) \otimes H^*(BG; \mathbb{F}_p),$$ such that the first map is induced by the map $x \mapsto x^{\otimes a}$ and the second map is defined by counting the number of points in $0$-dimensional moduli spaces that are $G$-equivariantly parametrised of order $a$, in the usual way. Some examples of such operations include quantum Steenrod operations (\cite[Section 4]{Seidel-Wilkins}, when $G = \mathbb{Z}/p$ and $a = p$), those involved in the quantum Cartan relation (\cite[Section 5]{wilkins}, when $G =\mathbb{Z}/p$ and $a=2p$), and those involved in the quantum Adem relation (\cite[Section 7]{wilkins}, when $G=D_{2p}$ or $G=S_p$ and $a=p^2$).

An immediate consequence of Theorem \ref{thm:main} is then the following:

\begin{corollary}
    \label{cor: ties_with_quantum}

    Any operadic $\mathbb{Z}/p$-equivariant quantum operation of order $p$ is a sum of quantum Steenrod operations and some other quantum operation defined by counting (nonequivariant) Gromov-Witten invariants.
\end{corollary}

\bibliographystyle{plain}
\bibliography{refs}
\end{document}